\documentclass{aic}

\usepackage{amsmath,amssymb,enumerate,mathtools,amsthm}
\newtheorem{theorem}{Theorem}[section]
\newtheorem{claim}{}[theorem]
\newtheorem{lemma}[theorem]{Lemma}

\newtheorem{corollary}[theorem]{Corollary}
\newtheorem{conjecture}[theorem]{Conjecture}

\newcommand{\bF}{\mathbb F}

\newcommand{\cE}{\mathcal{E}}

\newcommand{\cH}{\mathcal{H}}
\newcommand{\cM}{\mathcal{M}}

\newcommand{\cO}{\mathcal{O}}

\DeclareMathOperator{\cl}{cl}

\DeclareMathOperator{\PG}{PG}

\DeclareMathOperator{\AG}{AG}
\newcommand{\AGpp}{\AG^{\circ}}

\newcommand{\ip}[2]{\left\langle #1,#2 \right\rangle}
\newcommand{\twist}[2]{#1 \rtimes #2}
\newcommand{\twistw}[3]{#1 \rtimes_{#3} #2}
\newcommand{\floor}[1]{\left\lfloor #1 \right\rfloor}

\DeclareMathOperator{\BB}{BB}
\aicAUTHORdetails{%
  title = {The Structure of Binary Matroids with no Induced Claw or Fano Plane Restriction}, 
  author = {Marthe Bonamy, Franti\v{s}ek Kardo\v{s}, Tom Kelly, Peter Nelson, and Luke Postle},
  plaintextauthor = {Marthe Bonamy, Frantisek Kardos, Tom Kelly, Peter Nelson, Luke Postle},
  keywords = {matroid, induced.},
}   

\aicEDITORdetails{%
   year={2019},
   number={1},
   received={11 June 2018},   
   published={30 October 2019},  
   doi={10.19086/aic.10256},       
}   

\begin{document}

\begin{frontmatter}[classification=text]

\title{The Structure of Binary Matroids with no Induced Claw or Fano Plane Restriction} 

\author[Bonamy]{Marthe Bonamy}
\author[Kardos]{Franti\v{s}ek Kardo\v{s}}
\author[Kelly]{Tom Kelly}
\author[Nelson]{Peter Nelson}
\author[Postle]{Luke Postle}

\begin{abstract}
An \emph{induced restriction} of a simple binary matroid $M$ is a restriction $M|F$, where $F$ is a flat of $M$. We consider the class $\cM$ of all simple binary matroids $M$ containing neither a free matroid on three elements (which we call a \emph{claw}), nor a Fano plane as an induced restriction. We give an exact structure theorem for this class; two of its consequences are that the matroids in $\cM$ have unbounded critical number, while the matroids in $\cM$ not containing the clique $M(K_5)$ as an induced restriction have critical number at most $2$.
\end{abstract}
\end{frontmatter}




\section{Introduction}        

This paper considers questions from the theory of induced subgraphs in the setting of simple binary matroids. We adopt a nonstandard formalism of binary matroids, incorporating the notion of an `ambient space' into the definition of a matroid itself. A \emph{simple binary matroid} (hereafter just a \emph{matroid}) is a pair $M = (E,G)$, where $G$ is a finite-dimensional binary projective geometry and the \emph{ground set} $E$ is a subset of (the points of) $G$. The \emph{dimension} of $M$ is the dimension of $G$. If $E = G$, then we refer to $M$ itself as a \emph{projective geometry}. If no proper subgeometry of $G$ contains $E$, then $M$ is \emph{full-rank}. 

We say a matroid $M = (E,G)$ contains a matroid $N = (F,G')$ as a \emph{restriction} if there is a linear injection $\varphi\colon G' \to G$ for which $\varphi(F) \subseteq E$. If $\varphi(F) = E \cap \varphi(G')$, then we say $M$ contains $N$ as an \emph{induced restriction} and we refer to $\varphi$ as an \emph{induced embedding} of $N$ in $M$. If $\varphi$ is an induced embedding and also a bijection, then $M$ and $N$ are \emph{isomorphic}. Given this notion of isomorphism, it follows from the uniqueness of representability of binary matroids that matroids by our definition are simple binary matroids in the usual sense, except we allow our matroids to have a ground set not spanning $G$; a matroid whose ground set does not span $G$ is analogous to a graph with isolated vertices.

The \emph{critical number} $\chi(M)$ of an $n$-dimensional matroid $M = (E,G)$, is the smallest integer $k \ge 0$ such that $G-E$ contains a projective geometry of dimension $n-k$. This parameter (also called \emph{critical exponent}) is a matroidal version of chromatic number --  for example, if $M(G)$ is the cycle matroid of a graph $G$, then $\chi(M(G)) = \lceil \log_2 \chi(G) \rceil$ -- and can also be defined in terms of cocycles or the Tutte polynomial. See (\cite{oxley} p. 588) for a discussion and further reading, or \cite{kung} for a detailed treatment.

For each $t \ge 1$, write $\PG(t-1,2)$ as shorthand for the matroid $(E,G)$ where $E = G \cong \PG(t-1,2)$, and let $I_t$ denote the unique $t$-dimensional, $t$-element  matroid of full rank (its ground set is a `basis').  Write $F_7$ for the Fano matroid $\PG(2,2)$.  Let $K_5$ denote the matroid $(E,G)$ where $G \cong \PG(3,2)$ has elements naturally identified with the nonzero vectors in $\bF_2^4$, and $E$ comprises all vectors whose Hamming weight is $1$ or $2$. (In more traditional language, this is the cycle matroid of the five-vertex complete graph.) We prove the following. 

\begin{theorem}\label{main1}
	If $M$ is a simple binary matroid containing neither $I_3, F_7$, nor $K_5$ as an induced restriction, then $\chi(M) \le 2$. 
\end{theorem}

In fact, we show much more. Let $\cE_3$ denote the class of matroids whose induced restrictions of dimension at least $3$ all have ground sets of even size; we call these the \emph{even plane matroids}. (Note that in our language, it is meaningful to discuss three-dimensional induced restrictions that are not full-rank; we emphasise that our definition insists that even these `rank-deficient' induced restrictions have even size.) We will consider this class, which includes $K_5$ as well as all affine geometries, in some detail in what follows. Our main result gives a qualitative generalisation of Theorem~\ref{main1} by bounding $\chi$ for the induced $I_3,F_7$-restriction-free matroids that also omit $N$, for an arbitrary fixed $N \in \cE_3$. 

\begin{theorem}\label{main2}
	For every matroid $N \in \cE_3$, if $M$ is a simple binary matroid containing neither $I_3$, $F_7$, nor $N$ as an induced restriction, then $\chi(M) \le \dim(N) + 4$. 
\end{theorem}

This theorem is best-possible in the sense that $\chi$ is unbounded when no such $N$ is excluded, as the next theorem states. 

\begin{theorem}\label{gsfalse}
	For every integer $k$, there is a matroid $M \in \cE_3$ containing neither $I_3$ nor $F_7$ as an induced restriction such that $\chi(M) \ge k$. 
\end{theorem}

Theorems~\ref{main1} and~\ref{main2} are proved by means of an exact structure theorem for matroids with no induced $F_7$ or $I_3$-restrictions. The statement is slightly technical and requires two more definitions. If $M = (E,G)$ is an $n$-dimensional matroid, where $G$ is identified with $\bF_2^{n}-\{0\}$, then the \emph{doubling} of $M$ is the $(n+1)$-dimensional matroid $M' = (E',G')$, where $G'$ is identified with $\bF_2^{n+1}-\{0\}$ and $E' = \{0,1\} \times E$. Variants of this definition appear in many contexts in matroid theory; see \cite{ow} and \cite{whittle}, for example. We will see in Lemma~\ref{doublinggood} that doubling does not change critical number, nor the presence of $I_3$ or $F_7$ as induced restrictions. For an integer $t \ge 2$, write $\AGpp(t-1,2)$ for the matroid $((G-H) \cup \{x\},G)$, where $G \cong \PG(t-1,2)$ and $x$ is an element of a hyperplane $H$ of $G$. 

\begin{theorem}\label{maintech}
	If $M$ is a full-rank matroid, then $M$ does not have $I_3$ or $F_7$ as an induced restriction if and only if either
	\begin{itemize}
		\item $M \in \cE_3$, or
		\item there is some $t \ge 3$ such that $M$ arises from $\AGpp(t-1,2)$ by a sequence of doublings. 
	\end{itemize}
\end{theorem}

The first outcome will be substantially refined by Theorem~\ref{e3structure}, which is too technical to state here, but will give a description of all matroids in $\cE_3$ in terms of a decomposition.

\subsection*{Graph Theory}

Our work is motivated by the theory of induced subgraphs. In our setting, matroids of the form $I_t$ are maximal subject to containing no `circuits', so, despite being much simpler objects than trees, they play the role that trees do in graph theory. Our title borrows some graph terminology by referring to $I_3$ as a \emph{claw} (the automorphism group of $I_3$ acts symmetrically on its elements, just as that of $K_{1,3}$ acts symmetrically on its edges). Projective geometries correspond to cliques in the sense of being maximally dense and providing a trivial lower bound on critical number.  (We have $\chi(\PG(t-1,2)) = t$ and so any matroid containing $\PG(t-1,2)$ must have $\chi \ge t$.) The following celebrated graph theory conjecture states that excluding a clique and an induced tree bounds the chromatic number. 

\begin{conjecture}[{Gy\'arf\'as-Sumner  \cite{g85},\cite{s81}}]\label{gsc}
	For every tree $T$ and complete graph $K$, there is an integer $c$ such that if $G$ is a graph containing neither $T$ nor $K$ as an induced subgraph, then $\chi(G) \le c$. 
\end{conjecture}

Using the analogies discussed above, one could easily formulate a matroidal version of this conjecture, asserting that $\chi$ should be bounded by a function of $s$ and $t$ whenever we exclude a matroid $I_s$ and a projective geometry $\PG(t-1,2)$ as induced restrictions. Howeover, Theorem~\ref{gsfalse} refutes such a statement even in the seemingly innocuous case when $s=t=3$. The analogous case of Conjecture~\ref{gsc} is when $T$ and $K$ both have four vertices, a special case proved in \cite{g85}. However, Conjecture~\ref{gsc} itself remains open in general despite being known to hold for many different $T$, (\cite{g85} contains a proof when $T$ is a path or star, and \cite{css17} discusses many more resolved cases). Its failure in the setting of matroids is surprising. Despite this, we conjecture that when $t = 2$, we do recover a bound on $\chi$.

\begin{conjecture}
	For all $s \ge 1$ there exists an integer $k$ such that if $M$ is a simple binary matroid with no induced $I_s$-restriction or  $\PG(1,2)$-restriction, then $\chi(M) \le k$. 
\end{conjecture}

This conjecture is easy to prove for $s \in \{1,2,3\}$ -- we show in Corollary~\ref{recag2}  that $k = 1$ will do --  but will require new ideas for larger $s$. Its graph-theoretic analogue, where one excludes an induced tree $T$ and a triangle, is easy when $T = K_{1,t}$ (as the hypotheses impose a bound on maximum degree, which in turns bounds chromatic number), but remains open in general.

\subsection*{Even matroids} 

The class $\cE_3$, of matroids whose ($\ge 3$)-dimensional induced restrictions all have even size, arose for us in a negative way, as a source of counterexamples to an attempted proof of the analogue of Conjecture~\ref{gsc}. However, this class and related ones turn out to enjoy properties that we believe to be of substantial interest, and a good amount of our work concerns these classes for their own sake. We discuss our results, which we prove in Sections~\ref{evensection} and~\ref{evenplanesection}, briefly here. 

While the criterion for membership in $\cE_3$ is restrictive, the class appears relatively rich. It is not hard to prove (see Theorem~\ref{evenstructure}) that the definition is equivalent to the assertion that the induced restrictions of dimension \emph{exactly} $3$ all have even size. We show in Lemma~\ref{chibound} that $\cE_3$ is infinite, and that, in fact, its members can have critical number up to roughly half their dimension. We show as well that $\cE_3$ has the curious property of being an inclusion-minimal class that is closed under induced restrictions for which $\chi$ is unbounded. It seems difficult to construct a nontrivial class of graphs having an analogous property. 

\begin{theorem}\label{omnivorous}
	For every proper subclass $\cM$ of $\cE_3$ that is closed under induced restrictions, there exists $k$ so that $\chi(M) \le k$ for all $M \in\cM$.
\end{theorem}

We also define and consider generalizations of $\cE_3$; for each integer $k$, let $\cE_k$ denote the class of matroids whose induced restrictions of dimension at least $k$ all have even size. We will show, as we did with $\cE_3$, that the `at least' in the definition can be replaced with `exactly', and will give in Theorem~\ref{evenstructure} a characterisation of $\cE_k$ in terms of an iterated construction using matroids in $\cE_{k-1}$; this construction is particularly pleasant for $\cE_3$. For all the above reasons, we believe that the classes $\cE_k$ (as well as the classes $\cO_k$ of complements of matroids in $\cE_k$, comprising the matroids whose $k$-dimensional induced restrictions all have \emph{odd} size), will play an important role in the theory of induced restrictions in matroids. In the vein of Theorems~\ref{gsfalse} and~\ref{maintech}, we expect them to appear as basic classes in structure theorems, and, since they have unbounded critical number but omit small projective geometries and many other natural matroids, that they will be a good source of counterexamples.

\section{Preliminaries} 

Our matroid terminology is mostly standard (see, for example, \cite{oxley}), with a few adaptations to suit our adjusted formalism. Let $M = (E,G)$ be a matroid.  We call $E$ the \emph{ground set} of $G$. We occasionally write $E(M)$ for $E$ and $G(M)$ for $G$ where these sets have not yet been defined. The \emph{size} of $M$ is $|E|$, and $M$ is \emph{empty} if its size is zero. We do not refer to a `flat' of a matroid $M$ itself, only to a \emph{flat of $G$}, which simply means a projective geometry contained in $G$. Each flat $F$ has a \emph{dimension}, denoted $\dim(F)$, which is its dimension as a projective geometry, and a \emph{codimension} $\dim(G)-\dim(F)$ in $G$. The empty set is a flat of dimension zero. Flats of dimension $2,3$ and $\dim(G)-1$ are called \emph{triangles}, \emph{planes} and \emph{hyperplanes} respectively. A \emph{triangle-free} matroid is one whose ground set contains no triangles, or equivalently one with no $\PG(1,2)$-restriction.

For a flat $F$ of $G$, we write $M|F$ for the induced restriction $(E \cap F,F)$ of $M$. The \emph{closure} of a set $X \subseteq G$, written $\cl(X)$, is the unique minimal flat containing $X$, and the \emph{rank} of $X$, written $r(X)$, is the dimension of this flat. The rank of $M$ itself is defined by $r(M) = r(E)$; note that $r(M) \le \dim(M)$, recalling that $\dim(M)$ is defined to be $\dim(G)$.  If $r(M) = \dim(M)$ then we say $M$ is \emph{full-rank}. A \emph{basis} of $M$ is a minimal subset of $E$ whose rank is $r(M)$.  

It is often helpful to think of the elements of $G \cong \PG(n-1,2)$ as the nonzero vectors in $\bF_2^n$. In this way, we can fix a given basis of $G$ to correspond to the set of standard basis vectors, and the flats of $G$ correspond to the subspaces of $\bF_2^n$ with the zero element removed. In particular, a triangle of $G$ is a triple $\{x,y,z\}$ of distinct nonzero vectors with sum zero. Evoking this notationally, for elements $x$ and $y$ of $G \cup \{0\} = \bF_2^n$, we write $x + y$ for the  element corresponding to the vector sum of $x$ and $y$. Note that if $x \ne y$ then $x+y \ne 0$ is the third element in the triangle containing $x$ and $y$. If $F$ is a flat of $G$ and $w \in G-F$, then $\cl(F \cup \{w\})$ is the disjoint union of the three sets $\{w\},F$ and $F + w = \{x+w\colon x \in F\}$.  If $F$ is a hyperplane, then these three sets partition $G$. 

Still thinking of elements as vectors, we define a \emph{circuit} of $G$ as a minimal set of its elements $G$ whose sum is zero. Note that a triangle is a $3$-element circuit.

The \emph{complement} of $M$ is the matroid $\overline{M} = (G-E,G)$. For $n \ge k \ge 0$, the \emph{$n$-dimensional Bose-Burton Geometry of order $k$}, denoted $\BB(n-1,2;k)$, is the matroid $(G-F,G)$, where $F$ is a $k$-codimensional flat of $G \cong \PG(n-1,2)$. These are named for the authors of \cite{bb}, which proves they are the unique densest matroids omitting a $\PG(k,2)$-restriction. A Bose-Burton geometry of order zero is empty; one of order $1$ is an \emph{affine geometry}, for which we write $\AG(n-1,2)$, and one of order $n$ is a projective geometry. 
Bose-Burton and affine geometries have the following easy characterisation, which is implicit in \cite{bb}. 
\begin{lemma}\label{recognisebb}
	If $M = (E,G)$ is a matroid such that no triangle $T$ of $G$ satisfies $|T \cap E| = 1$, then $M$ is a Bose-Burton geometry. If, additionally, no triangle $T$ of $G$ is contained in $E$, then $M$ is either empty or an affine geometry. 
\end{lemma}
\begin{proof}
	The condition given implies that for all $x,y \in G-E$ we have $x + y \in G-E$. Viewing the elements of $G$ as the nonzero vectors in $\bF_2^n$, we see that $(G-E) \cup \{0\}$ is a subspace, so $G-E$ is a flat of $G$ and thus $M$ is a Bose-Burton geometry. If $G-E$ has codimension at least $2$, then $E$ clearly contains a triangle of $G$; the second part follows. 
\end{proof}

We freely use the this, as well as its  complementary statement, which implies that $(E,\cl(E))$ is a projective geometry for any matroid $(E,G)$ where no triangle of $G$ contains exactly two elements of $E$. 

\subsection*{Critical Number}
Recall that the \emph{critical number} $\chi(M)$ is the minimum codimension of a projective geometry restriction of $\overline{M}$. This parameter plays the role of chromatic number. We trivially have $\chi(\BB(n-1,2;k)) = k$, and in fact, $\chi(M)$ for a general matroid $M$ is just the smallest order of a Bose-Burton geometry containing $M$ as a restriction; thus, Bose-Burton geometries are analogous to complete multipartite graphs. 

We freely use a few easy facts about $\chi$; for example, $\chi(M') \le \chi(M)$ for any restriction $M'$ of $M$; we have $\chi(M) \le r(M)$ for all $M$; and if $M_i = (E_i,G)$ for $i \in \{1,2\}$ while $M = (E_1 \cup E_2,G)$, then $\chi(M) \le \chi(M_1) + \chi(M_2)$. This gives that adding an element to a matroid increases its critical number by at most $1$, and since the complement of a hyperplane has critical number at most $1$, we have $\chi(M|H) \le \chi(M) \le \chi(M|H) + 1$ whenever $M = (E,G)$ and $H$ is a hyperplane of $G$.

\subsection*{Doublings}

Let $M = (E,G)$ be a matroid, let $H$ be a hyperplane of $G$, and let $w \in G-H$. We say that $M$ is the \emph{doubling of $M|H$ by $w$} if $w \notin E$ and $E = (E \cap H) \cup (w + (E \cap H))$. This condition is equivalent to the statement that $w \in G-(E \cup H)$ while every triangle $T$ of $G$ containing $w$ satisfies $|T \cap E| \in \{0,2\}$ (In fact, if this condition holds then $M$ is the doubling of $M|H'$ for \emph{every} hyperplane $H'$ not containing $w$). Note that this definition agrees with the more concrete one given in the introduction. 

It is easy to check that, for a matroid $M_0$, any two doublings of $M_0$ are isomorphic. When $w$ is unimportant, we refer to $M$ just as the \emph{doubling} of $M_0$. Doublings preserve the property of being a Bose-Burton geometry; the doubling of $\BB(n-1,2;k)$ is $\BB(n,2;k)$. The following lemma shows that they also preserve critical number and the absence of most fixed induced restrictions.

\begin{lemma}\label{doublinggood}
	Let $M$ be the doubling of a matroid $M_0$. Then
	\begin{itemize}
		\item $\chi(M) = \chi(M_0)$, and 
		\item if $N$ is a matroid that is not a doubling of another matroid, and $M_0$ contains no induced $N$-restriction, then neither does $M$.
	\end{itemize} 
\end{lemma}
\begin{proof}
	Let $M = (E,G)$ and let $w$ and $H$ be a point and a hyperplane of $G$ so that $M_0 = (E \cap H,H)$ and each triangle of $G$ containing $w$ contains either two or zero elements of $M$. 
	
	To prove the first part, let $c = \chi(M_0)$ and let $U_0$ be a $c$-codimensional flat of $H$ for which $U_0 \cap (E \cap H)$ is empty. Since $M_0$ is a restriction of $M$, we have $\chi(M) \ge c$. However, by construction the flat $U = \cl(U_0 \cup \{w\})$ is disjoint from $E$ and has codimension $c$ in $G$. Therefore $\chi(M) \le c$, and the result holds. 
	
	For the second part, let $d = \dim(N)$ and let $F$ be a $d$-dimensional flat of $G$. If $w \in F$, then every triangle of $F$ containing $w$ contains zero or two elements of $E$, so $M|F$ is a doubling and thus $M|F \not\cong N$. If $w \notin F$ is a $d$-dimensional flat of $G$ not containing $w$, then $M|F \cong M|F'$, where $F' = H \cap (\cl(F \cup \{w\})$. Since $M|F'$ is an induced restriction of $M_0$, it follows that $M|F \not\cong N$, so $M$ has no induced $N$-restriction. 
\end{proof}

In particular, the above lemma applies whenever $N$ has odd size. 

\section{Even matroids and twist doublings}\label{evensection} 

Let $k \ge 2$ be an integer. We say that a matroid $M$ is \emph{$k$-even} if every induced restriction of $M$ of dimension at least $k$ has even size. Let $\cE_k$ denote the class of $k$-even matroids. 

Let $M = (E,G)$ and $N = (D,G)$ be matroids, let $G'$ be a projective geometry having $G$ as a hyperplane, and let $w \in G' - G$. The \emph{twist doubling of $M$ by $w$ with respect to $N$} is the matroid $M' = (E',G')$ where $E' = E \cup (w + (E \Delta D))$. Note that $w \notin E'$ and that $M'|G = M$. 

We write $\twistw{M}{N}{(G',w)}$ for this matroid. Note that this matroid is determined up to isomorphism by just $M$ and $N$; we simply write $\twist{M}{N}$ when $G'$ and $w$ are not important. If $N$ is empty, then $\twist{M}{N}$ is just the doubling of $M$. For any matroid $M' = (E',G')$, any $w \in G'-E'$, and any hyperplane $G$ of $G'$ not containing $w$, it is not hard to see that $M' = \twistw{(M'|G)}{N}{(G',w)}$, where \[N = (\{x \in G \colon |\{x,x+w\} \cap E'| = 1\},G).\] Thus, every matroid that is not a projective geometry is a twist doubling of a smaller matroid. The next lemma characterises exactly when a twist doubling preserves membership in $\cE_k$.

\begin{lemma}\label{dothetwist}
	If $k \ge 2$ is an integer, and $M = (E,G)$ and $N = (D,G)$ are matroids with $M \in \cE_k$, then $\twist{M}{N} \in \cE_k$ if and only if $N \in \cE_{k-1}$. 
\end{lemma}
\begin{proof}
	Let $G'$ be a projective geometry containing $G$ as a hyperplane, let $w \in G'-G$, and let $M' = (E',G') = \twistw{M}{N}{(G',w)}$.
	
	Suppose first that $N \in \cE_{k-1}$. Let $F$ be a flat of $G'$ of dimension $h \ge k$. Let $F'$ be an $(h+1)$-dimensional flat containing $F \cup \{w\}$ and let $F_0 = F' \cap G$. Note that $F_0$ is a hyperplane of $F'$. Identify the elements in $F'$ with the nonzero vectors in $\bF_2^{h+1}$. Write $\ip{\cdot}{\cdot}\colon (\bF_2^{h+1})^2 \to \bF_2$ for the `dot product' bilinear form on this vector space, and let $u \in \bF_2^{h+1}$ be the vector for which $F = \{v \in \bF_2^{h+1}\colon \ip{u}{v} = 0\}$. Note that $F'$ is the disjoint union of the pairs $\{x,x+w\}$ where $x \in F_0$, and that $N|F_0 \in \cE_{k-1}$ and $M|F_0 \in \cE_k$ by assumption. 
	
	For each set $X \subseteq G'$, let $1_X\colon G' \to \bF_2$ be the characteristic function of $X$, so $1_F(x) = \ip{x}{u} + 1$. By assumption we have $1_{E'}(w) = 0$ while $1_{E'}(x+w) = 1_{E \Delta D}(x) = 1_E(x)+1_D(x)$ for each $x \in F_0$. The expression $|E' \cap F| \pmod{2}$ is given by
	\begin{align*}
	& \sum_{x \in F'}1_{E'}(x)1_F(x) \\
	&= \sum_{x \in F_0}(1_E(x)1_F(x) + 1_{E'}(x+w)1_F(x+w))\\
	&= \sum_{x \in F_0}(1_E(x)(\ip{x}{u}+1) + (1_E(x) + 1_D(x))(\ip{x+w}{u}+1))\\
	&= \sum_{x \in F_0}(1_E(x)\ip{w}{u} + 1_D(x)(\ip{x}{u} + \ip{w}{u} + 1))\\
	&= \ip{w}{u}\sum_{x \in F_0}1_E(x) + \ip{w}{u}\sum_{x \in F_0'}1_D(x)  + \sum_{x \in F_0}1_F(x)1_D(x)\\
	&= \ip{w}{u}|E \cap F_0| + \ip{w}{u}|D \cap F_0| + |D \cap (F \cap F_0)|.
	\end{align*}
	Now $\dim(F_0) = h \ge k$ and $M \in \cE_k$ implies that the first term is even, and $\dim(F_0 \cap F) \ge h-1 \ge k-1$ and $N|F_0 \in \cE_{k-1}$ together imply that the last two terms are even; thus $|E' \cap F|$ is even. This gives $M' \in \cE_k$ as required.
	
	Conversely, suppose that $M' \in \cE_k$. Note that for each $x \in G$, the set $\{x,x+w\} \cap E'$ has size $1$ if and only if $x \in D$. If $N \notin \cE_{k-1}$ then $G$ has a flat $F$ with $\dim(F) \ge k-1$ and $|F \cap D|$ odd. But then \[|\cl(F \cup \{w\}) \cap E'| = \sum_{x \in F - \{w\}}(|\{x,x+w\} \cap E'|) \equiv |F \cap D| \equiv 1 \pmod{2}.\] Since $\dim(\cl(F \cup \{w\})) \ge k$, this contradicts $M' \in \cE_k$. 
\end{proof}

The above implies that every matroid in $\cE_k$ that is not a projective geometry is a twist doubling $\twist{M}{N}$ where $M \in \cE_k$ and $N \in \cE_{k-1}$. This has as a consequence the following description of matroids in $\cE_k$. 

\begin{theorem}\label{evenstructure}
	Let $k \ge 2$ be an integer and $M$ be a matroid of dimension $n$. The following are equivalent:
	\begin{enumerate}
		\item\label{em1} $M \in \cE_k$,
		\item\label{em2} every $k$-dimensional induced restriction of $M$ has even size, 
		\item\label{em3} there are matroids $M_0,M_1, \dotsc, M_s$ and $N_0, \dotsc, N_{s-1} \in \cE_{k-1}$ such that $\dim(M_0) \le  k-1$ and $M_s = M$, while $G(M_i) = G(N_i)$ and $M_{i+1} \cong \twist{M_i}{N_i}$ for each $i \in \{0, \dotsc, s-1\}$.
	\end{enumerate}
\end{theorem}
\begin{proof}
	That (\ref{em1}) implies (\ref{em2}) is immediate. Suppose that (\ref{em2}) holds and (\ref{em1}) does not. Let $M' = (E',G')$ be a minimal induced restriction of $M$ for which $M' \notin \cE_k$ and let $n' = \dim(M')$; note that $n' > k$. Let $\cH$ be the collection of hyperplanes of $G'$. By minimality, for each $H \in \cH$ the quantity $|H \cap E'|$ is even, while $|E'|$ is odd. Each $x \in G'$ is in $2^{n-2}-1$ hyperplanes in $\cH$, so $\sum_{H \in \cH}|E' \cap H| = (2^{n-2}-1)|E'|$. Since $2^{n-2}-1$ is odd and each $|E' \cap H|$ is even, it follows that $|E'|$ is even, a contradiction. Therefore (\ref{em1}) and (\ref{em2}) are equivalent. 
	
	Since every matroid of dimension at most $k-1$ is in $\cE_k$, it follows from Lemma~\ref{dothetwist} that (\ref{em3}) implies (\ref{em1}). The converse implication clearly holds if $n = \dim(M)$ is at most $k-1$; suppose inductively that $n \ge k$ and that it holds for all dimensions less that $n$. Let $M = (E,G) \in \cE_k$. If $E = G$ then $|E| = 2^d-1$ is odd so the fact that $n \ge k$ contradicts $M \in \cE_k$. Otherwise, let $w \in G - E$. Let $H$ be a hyperplane of $G$ not containing $w$. Now we have $M = \twistw{(M|H)}{N}{w}$ for the matroid $N = (\{x \in H\colon |\{x,x+w\} \cap E| = 1\},H)$; it follows from Lemma~\ref{dothetwist} that $N \in \cE_{k-1}$. Since $M|H \in \cE_{k-1}$, (\ref{em3}) follows routinely by induction. 
\end{proof}

We are mainly concerned from here on with the class $\cE_3$. It follows from Lemma~\ref{recognisebb} that every matroid $N$ in $\cE_2$ is empty or an affine geometry. If $N$ is empty, then a twist doubling by $N$ is simply a doubling. If $N$ is an affine geometry, we call the twist doubling of a matroid $M = (E,G)$ by $N$ a \emph{semidoubling} of $M$, since the set of points that are `doubled' forms a hyperplane $H$ of $G$, while the other points are `antidoubled'. This hyperplane $H$ is simply the complement of $E(N)$, and it is more convenient in this case to describe the semidoubling in terms of $H$ rather than introducing a matroid $N$; if $H$ is a hyperplane of $G$ then the \emph{semidoubling of $M$ with respect to $H$} refers to the twist doubling of $M$ with respect to the affine geometry $N = (G-H,G)$. 
Lemma~\ref{dothetwist} has the following consequences for matroids in $\cE_3$.

\begin{corollary}\label{e3grow}
	Let $M = (E,G) \in \cE_3$ and $H$ be a hyperplane of $G$. If $G-H \not\subseteq E$, then $M$ is a doubling or semidoubling of $M|H$.
\end{corollary}
\begin{proof}
	Let $w \in G - (E \cup H)$. There is a matroid $N = (F,H)$ for which $M = \twistw{(M|H)}{N}{w}$; Lemma~\ref{dothetwist} gives $N \in \cE_2$ and the result follows. 
\end{proof}

\begin{corollary}\label{weake3structure}
	Let $M$ be a matroid of dimension at least $2$. Then $M \in \cE_3$ if and only if $M$ arises from a $2$-dimensional matroid by a sequence of doublings and/or semidoublings. 
\end{corollary}
\begin{proof}
	Since a doubling/semidoubling is precisely a twist doubling by a matroid in $\cE_2$, the backwards implication follows from Lemma~\ref{dothetwist}. Conversely, if $M = (E,G) \in \cE_3$ we may inductively assume that $\dim(M) \ge 3$ and therefore $E \ne G$, since otherwise $M$ has an induced $F_7$-restriction. Let $w \in G-E$ and $H$ be a hyperplane of $G$ not containing $w$. Corollary~\ref{e3grow} now gives that $M$ is a doubling or semidoubling of $M|H$, and the result follows by induction. 
\end{proof}

\section{Even plane matroids}\label{evenplanesection}

Call the members of $\cE_3$ the \emph{even plane matroids}. We have seen that they arise from small matroids via doublings and semidoublings; we now show that nearly every example is in fact a semidoubling.

\begin{lemma}\label{findsemidoubler}
	If $M = (E,G) \in \cE_3$ and $H$ is a hyperplane of $G$ for which $M|H$ is not a Bose-Burton geometry, then $M$ is a semidoubling of $M|H$. 
\end{lemma}
\begin{proof}
	Let $M = (E,G)$. Since $M|H$ is not a Bose-Burton geometry, there is a triangle $T_0$ of $H$ containing precisely one element of $E$. Let $P$ be a plane of $G$ containing $T_0$ but not contained in $H$. Since $M \in \cE_3$ we have $|(P-T_0) \cap E|$ odd; it follows easily that there is a triangle $T = \{u,v,w\} \subseteq P$ such that $u \in H$ while $v,w \notin H$, and $T \cap E = \{v\}$. Now $w \in G-(H \cup E)$ and the triangle $T$ certifies that $M$ is not a doubling of $M|H$ by $w$; the fact that $M$ is a semidoubling of $M|H$ by $w$ follows from Corollary~\ref{e3grow}. 
\end{proof}

\begin{corollary}\label{findsemidoubler1}
	If $M = (E,G) \in \cE_3$ and $\chi(M) \ge 3$, then for every hyperplane $H$ of $G$, the matroid $M$ is a semidoubling of $M|H$.
\end{corollary}
\begin{proof}
	Let $M = (E,G)$. We may assume by Lemma~\ref{findsemidoubler} that $M|H$ is a Bose-Burton geometry, so $F = H-E$ is a flat of $H$, and $\dim(F) \ge \dim(H) - 2$, as otherwise $H-F$ has an $F_7$-restriction contained in $E$. If $\dim(F) \ge \dim(H)-1$ then $\chi(M|H) \le 1$ and so $\chi(M) \le 2$, a contradiction. Therefore $\dim(F) = \dim(H)-2$. If $G-H \subseteq E$ then $M$ is the complement of the $3$-codimensional flat $F$ in $G$, so $M$ has an $F_7$-restriction. Thus there exists $w \in G - (H \cup E)$. If $G$ has a triangle $T = \{w,v,x\}$ where $x \in F$ and $T \cap E = \{v\}$ then $T$ certifies that $M$ is not a doubling of $M|H$ by $w$, and the result follows from Corollary~\ref{e3grow}. Therefore there is no such triangle; this implies that the flat $F'= \{w\} \cup F \cup (w + F) = \cl(F \cup \{w\})$ is disjoint from $E$. Now $\dim(F') = \dim(M) - 2$ and so $\chi(M) \le 2$, a contradiction. 
\end{proof}

This gives a decomposition theorem for $\cE_3$ that refines Theorem~\ref{evenstructure}.
\begin{theorem}\label{e3structure}
	Let $M$ be a nonempty matroid. Then $M \in \cE_3$ if and only if $M$ arises from a Bose-Burton geometry of order $1$ or $2$ by a sequence of semidoublings.
\end{theorem}
\begin{proof}
	Each Bose-Burton geometry of order at most $2$ and dimension at least $2$ arises from a dimension-$2$ matroid by a sequence of doublings, so the backwards implication follows directly from Corollary~\ref{weake3structure}. Conversely, let $M = (E,G) \in \cE_3$, and assume inductively that $\dim(M) > 2$ and the result holds for lower dimension. If $G$ has no triangle $T$ with $|T \cap E| = 1$, then $M$ is a Bose-Burton geometry, and since $M \in \cE_3$ the order is at most $2$, so $M$ satisfies the theorem. If $G$ has such a triangle $T$, let $H$ be a hyperplane of $G$ containing $T$; now $T$ certifies that $M|H$ is not a Bose-Burton geometry, and Lemma~\ref{findsemidoubler} gives that $M$ is a semidoubling of $M|H$. The result follows inductively. 
\end{proof}

\subsection*{Critical number}

We now analyse how critical number grows with semidoublings.

\begin{lemma}\label{pushchi}
	Let $M = (E,G)$ be a semidoubling of $M_0 = (E \cap G_0,G_0)$ with respect to a hyperplane $H_0$ of $G_0$. Then $\chi(M) = \chi(M|H_0) + 1$.  
\end{lemma}

\begin{proof}
	Let $k = \chi(M_0)$ and $n = \dim(G_0)$ and suppose that $M$ is the semidoubling of $M_0$ by $w$ with respect to $H_0$. Since $G_0$ is a hyperplane of $G$ and $H_0$ is a hyperplane of $G_0$, we have $\chi(M|H_0) \in \{k-1,k\}$ and $\chi(M) \in \{k,k+1\}$; it is enough to show that $\chi(M|H_0) \le k-1$ if and only if $\chi(M) \le k$. 
	
	Suppose that $\chi(M|H_0) \le k-1$. Let $F_0$ be a $(k-1)$-codimensional flat of $H_0$ for which $F_0 \cap E = \varnothing$. So $\dim(F_0) = (n-1)-(k-1) = n-k$. Therefore the flat $\cl(F_0 \cup \{w\})$ has dimension $n-k+1$; since $F_0 \subseteq H_0$ the definition of semidoubling implies that $\cl(F_0 \cup \{w\})$ is disjoint from $E$. Since $n-k+1 = \dim(G) - k$, it follows that $\chi(M) \le k$ as required. 
	
	Suppose, therefore, that $\chi(M) \le k$. 
	\begin{claim}
		There is a $k$-codimensional flat $F_2$ of $G$, containing $w$, such that $F_2 \cap E = \varnothing$. 
	\end{claim}
	\begin{proof}[Subproof:]
		Since $\chi(M) \le k$, there is a $k$-codimensional flat $F_2'$ of $G$ for which $F_2' \cap E = \varnothing$; we may assume that $w \notin F_2'$. Since $H = \cl(H_0 \cup \{w\})$ is a hyperplane of $G$, the flat $L = F_2' \cap H$ has dimension at least $\dim(F_2') -1$. Now the flat $L$ is disjoint from $E$ and is contained in $\cl(H_0 \cup \{w\})$, so by the definition of semidoubling, the flat $F_2 = \cl(L \cup \{w\})$ is also disjoint from $E$. But $w \notin F_2'$ and $L \subseteq F_2'$, so $w \notin L$ and therefore $\dim(F_2) \ge \dim(L) + 1 \ge \dim(F_2') = n+1-k$. The claim follows. 
	\end{proof}
	
	We now argue that $F_2 \subseteq \cl(H_0 \cup \{w\})$; if not, then $F_2$ contains a triangle containing $w$ and an element of $G_0-H_0$, but any such triangle intersects $H_0$ in exactly one element, so by the definition of semidoubling contains exactly one element of $E$, contradicting $F_2 \cap E = \varnothing$. Therefore $F_2 \subseteq \cl(H_0 \cup \{w\})$, so the flat $F_1 = F_2 \cap H_0$ has dimension $(n+1-k)-1 = n-k$, and is contained in $H_0$. Since $F_1$ is disjoint from $E$, this gives $\chi(M|H_0) \le \dim(M_0) - \dim(F_1) = (n-1)-(n-k) = k-1$, as required. 
\end{proof}

In particular, the above lemma applies when $M_0$ is itself a semidoubling or doubling of $M|H_0$. We use this fact to prove the next lemma, which gives a best-possible bound on $\chi$ for matroids in $\cE_3$. This lemma easily implies Theorem~\ref{gsfalse}.

\begin{lemma}\label{chibound}
	If $M \in \cE_3$ and $\dim(M) = n$, then $\chi(M) \le \floor{\tfrac{1}{2}n} + 1$. For all $n \ge 1$ there is some $M \in \cE_3$ where equality holds.
\end{lemma}
\begin{proof}
	The bound is trivial for $\dim(M) \le 1$; suppose that $\dim(M) > 1$ and it holds for matroids of smaller dimension. If $M$ is a semidoubling of some restriction $M|H$ with respect to a hyperplane $H_0$ of $H$, then Lemma~\ref{pushchi} gives $\chi(M) = \chi(M|H) + 1 \le \floor{\tfrac{1}{2}(n-2)} + 2 = \floor{\tfrac{1}{2}n} + 1$ as required. Otherwise, by Theorem~\ref{e3structure} the matroid $M$ is a Bose-Burton geometry of order $k \in \{1,2\}$, giving $\chi(M) = k = \floor{\tfrac{1}{2}k}+1 \le \floor{\tfrac{1}{2}n}+1$, so the bound holds.
	
	Projective geometries satisfy the bound with equality in dimension at most $2$; let $n \ge 2$ be even and suppose that equality holds for $M = (E,G) \in \cE_3$; i.e. $\chi(M) = \tfrac{1}{2}n + 1$. Let $H$ be a hyperplane of $G$ and let $M' = (E',G')$ be a semidoubling of $M$ with respect to $H$. Now $\dim(M') = n+1$ and Lemma~\ref{pushchi} gives $\chi(M') = \chi(M|H) + 1 \ge \chi(M) = \floor{\tfrac{1}{2}(n+1)} + 1$, so $M'$ satisfies the bound with equality. Let $M''$ be the semidoubling of $M'$ with respect to the hyperplane $G$ of $G'$. We have $\dim(M'') = n+2$ and Lemma~\ref{pushchi} gives $\chi(M'') = \chi(M) + 1 = \tfrac{1}{2}n+2 = \floor{\tfrac{1}{2}(n+2)} +1$, so $M''$ satisfies the bound with equality. An inductive argument implies that for all $n \ge 1$ the bound is satisfied with equality. 
\end{proof}

\subsection*{Universality}

We now prove that each $M \in \cE_3$ contains as an induced restriction every matroid $N \in \cE_3$ of dimension at most $\chi(M)-4$. This is done by proving the following stronger statement. 

\begin{lemma}
	Let $M = (E,G)$ and $N = (E',G')$ be matroids in  $\cE_3$ and let $H,H'$ be hyperplanes of $G,G'$ respectively. If $\chi(M) \ge \dim(N) + \min(4,\chi(N) + 2)$, then there is an induced embedding $\varphi$ of $N$ in $M$ for which $\varphi(G') \cap H = \varphi(H')$.
\end{lemma}
\begin{proof}
	We may assume that $\dim(N) \ne 0$, so $\chi(M) \ge 3$;  Corollary~\ref{findsemidoubler1} implies that there is some $w \in G-(H \cup E)$ for which $M$ is the semidoubling of $M|H$ by $w$ with respect to some hyperplane $H_0$ of $H$.

	If $\dim(N) = 1$ and $\chi(N) = 0$ then, since $w \notin E \cup H$ and $H' = \varnothing$, the map sending the single element of $G'$ to $w$ is the required induced embedding. We proceed by induction on $\dim(N) + \chi(N)$. In particular, we may assume that the lemma holds for $N|H'$. By Corollary~\ref{e3grow}, either $G'-H' \subseteq E'$, or $N$ is a doubling or semidoubling of $N|H'$. We distinguish these three cases. 
	
	\textbf{Case 1:} $G'-H' \subseteq E'$. 
	
	Let $N_0 = (E'-(G'-H'),G')$; note that $\chi(N_0) = \chi(N|H') = \chi(N)-1$, and that $N_0 \in \cE_3$, since every plane of $G'$ contains either zero or four elements of $G'-H'$. Every triangle $T$ of $H'$ is contained in a plane $P$ of $G'$ containing $|T \cap E'| + 4$ elements of $E'$, so we must have $|T \cap E'|$ even for every such triangle. Thus $N|H' \in \cE_2$, which gives $1 \ge \chi(N|H') = \chi(N_0)$ and $\chi(N) \le 2$.
	
	The matroid $M|H$ satisfies 
	\begin{align*}
	\chi(M|H) &\ge \chi(M)-1 \\
	&\ge \dim(N) + \min(4,\chi(N)+2)-1 \\
	& = \dim(N_0) + \chi(N) + 1 \\
	& = \dim(N_0) + \chi(N_0) + 2.
	\end{align*}
	Since $\chi(N_0) \le 1$, inductively there is an induced embedding $\varphi_0$ of $N_0$ in $M|H$ for which $\varphi(G') \cap H_0 = \varphi(H')$. Let $x \in G'-H'$ and let $\varphi\colon G' \to G$ be defined by 
	\[\varphi(x) = \begin{cases}\varphi_0(x), & \text{if    } x \in H', \\ w + \varphi_0(x), & \text{if    }  x \in G'-H'.\end{cases}\]
	Since $w \notin H$ and $\varphi(H') \subseteq H$, it is immediate that $\varphi$ is a linear injection that maps $H'$ to $H$ and $G'-H'$ to $G-H$. Since $N_0$ contains no elements of $G'-H'$ and $\varphi_0$ is an induced embedding of $N_0$ in $M|H$, for each $x \in G'-H'$ we have $\varphi_0(x) \notin E$. Since $M$ is a semidoubling of $M|H$ and $\varphi_0(x) \notin H_0$ for each such $x$, we thus have $E \ni w + \varphi_0(x) = \varphi(x)$. So $\varphi(G'-H') \subseteq E$. Since $G'-H' \subseteq E'$ while $\varphi$ agrees with $\varphi_0$ on $H'$, it follows that $\varphi$ is the required induced embedding of $N$ in $M$.

	\textbf{Case 2:} $N$ is a doubling of $N|H'$.
	
	Let $w' \in G'-(E' \cup H')$ be such that $N$ is the doubling of $N|H'$ by $w'$. By the definition of semidoubling, the hyperplane $H_0^+ = \cl(H_0 \cup \{w\})$ of $G$ is such that $M|H_0^+$ is the doubling of $M|H_0$ by $w$, so 
	\begin{align*} 
	\chi(M|H_0) &= \chi(M|H_0^+) \\
	& \ge \chi(M)-1 \\
	& \ge \dim(N)+ \min(4,\chi(N)+2) - 1 \\
	& \ge \dim(N|H') + \min(4,\chi(N|H') + 2).
	\end{align*}
	Therefore, inductively, there is an induced embedding $\varphi_0$ of $N|H'$ in $M|H_0$. Let $\varphi\colon G' \to G$ be the linear injection extending $\varphi_0$ for which $\varphi(w') = w$. Clearly $\varphi(G') \cap H = \varphi(H')$. Since $N$ is the doubling of $N|H'$ by $w'$ while $M|H_0^+$ is the doubling of $M|H_0$ by $w$, it is easy to show that $\varphi$ is the required induced embedding of $N$ in $M$. 
	
	\textbf{Case 3:} $N$ is a semidoubling of $N|H'$. 
	
	Let $w' \in G'-H'$ and $H_0$ be a hyperplane of $H_0'$ for which $N$ is the semidoubling of $N|H'$ by $w'$ with respect to $H_0'$. Note that 
	\begin{align*}
	\chi(M|H) &\ge \chi(M) - 1 \\
	&\ge \dim(N)-1 + \min(4,\chi(N) + 2) \\
	&\ge \dim(N|H') + \min(4,\chi(N|H')+2).
	\end{align*}
	Inductively, there is an induced embedding $\varphi_0$ of $N|H'$ in $M|H$ for which $\varphi(H) \cap H_0 = \varphi(H_0')$. Let $\varphi\colon G' \to G$ be the linear injection extending $\varphi_0$ for which $\varphi(w') = w$; we argue that $\varphi$ is the required induced embedding of $N$ in $M$. It is immediate that $\varphi(G') \cap H = H'$. 
	
	Let $u \in G'$. If $u \in H'$ then $\varphi(u) = \varphi_0(u)$, which is in $E$ if and only if $u \in F$. If $u = w'$ then $u \notin F$ and $\varphi(u) = w \notin E$. Otherwise we have $u = x' + w$ for some $x' \in H'$. Now $\varphi(u) = \varphi_0(x') + w$, which is in $E$ if and only if $\varphi_0(x') \notin H_0 \Delta E$. By the choice of $\varphi_0$ we have $\varphi_0(x') \notin H_0 \Delta E$ if and only if $x' \notin H_0' \Delta E'$, which holds if and only if $x' + w' \in E'$. So for all $u \in G'$ we have $u \in E'$ if and only if $\varphi(u) \in E$, as required. 
\end{proof}

The following result, which easily implies Theorem~\ref{omnivorous}, is a cleaner corollary of the above. 

\begin{theorem}\label{hungry}
	If $M,N \in \cE_3$ and $\chi(M) \ge \dim(N) + 4$, then $M$ has an induced $N$-restriction.
\end{theorem}

\section{The main results}

In this section we prove Theorems~\ref{main1} and ~\ref{maintech}. We first need three lemmas that recognise affine geometries and similar matroids in the claw-free case. 

\begin{lemma}\label{recag1}
	If $M = (E,G)$ is a triangle-free matroid in which every triple of distinct elements is contained in a four-element circuit, then $(E,\cl(E))$ is an affine geometry.
\end{lemma}
\begin{proof}
	The circuit condition given is equivalent to the statement that for all distinct $u,v,w \in E$, we have $u+v+w \in E$. Let $u_0 \in E$ and $W = E + u_0$, so $0 \in W$ and for all distinct nonzero $x,y \in W$ we have $x+y = ((x-u_0)+(y-u_0) + u_0)+u_0 \in W$ by the sum condition, so $W$ is a subspace of $\bF_2^n$. Therefore $(E,\cl(E))$ is a projective or affine geometry; the lemma follows from the fact that $M$ is triangle-free.
\end{proof}
\begin{corollary}\label{recag2}
	If $M = (E,G)$ is a triangle-free matroid with no induced $I_3$-restriction, then $(E,\cl(E))$ is an affine geometry. 
\end{corollary}
\begin{proof}
	Let $u,v,w \in E$ be distinct. Since $\{u,v,w\}$ is not a triangle and $M$ has no induced $I_3$-restriction, there is some $x \in \cl(\{u,v,w\}) \cap E$ other than $u,v,w$. Since $M$ is triangle-free, we have $x \notin \{u+v,u+w,v+w\}$ and so $x = u+v+w$. Therefore $\{u,v,w,x\}$ is a four-element circuit. The corollary follows from Lemma~\ref{recag1}.
\end{proof}

Recall that, for $s \ge 2$, $\AGpp(s-1,2)$ denotes the unique matroid obtained by adding a single element to the affine geometry $\AG(s-1,2)$. 

\begin{lemma}\label{recogniseagplus}
	Let $M$ be an $n$-dimensional matroid with an $\AG(n-1,2)$-restriction. If $M$ does not contain $F_7$ or $I_3$ as an induced restriction, then either $M \cong \AG(n-1,2)$, or there exists $s \ge 2$ such that $M$ arises from $\AGpp(s-1,2)$ by a sequence of doublings. 
\end{lemma}
\begin{proof}
	The result is trivial for $n = 1$; suppose that $n > 1$ and the theorem holds for smaller $n$. 
	
	Let $M = (E,G)$ and let $H$ be a hyperplane of $G$ so that $G-H \subseteq E$. Let $M_0 = M|H$. If $M_0$ contains an induced $I_3$-restriction then so does $M$. If $E \cap H$ contains a triangle $T$, then $\cl(T \cup \{w\})$ is an $F_7$-restriction of $M$ for any $w \in G-H$. Thus, Corollary~\ref{recag2} implies that $M_0' = (E \cap H,\cl(E \cap H))$ is an affine geometry. If $\dim(M_0') = 0$ then $M \cong \AG(n-1,2)$. If $\dim(M_0') = 1$ then $M \cong \AGpp(n-1,2)$. Otherwise, let $w$ be an element of the (nonempty) hyperplane $H_0$ of $H$ for which $E \cap H = H-H_0$; note that $w \notin E$. It is easy to see that every triangle of $G$ containing $w$ contains either two elements of $H-H_0$, two elements of $G-H$, or no elements of $E$. Thus $M$ is a doubling of some matroid $M_1$ by $w$. Since $w \in H$, the matroid $M_1$ is the restriction of $M$ to a hyperplane other than $H$, so $M_1$ has an $\AG(n-2,2)$-restriction. The result follows by induction. 
\end{proof}

We now prove the technical lemma from which our main structure theorem will follow; it states that when a hyperplane $H$ of an $(I_3,F_7)$-free matroid $(E,G)$ is partitioned according to the triangles containing some fixed $w \notin E \cup H$, then one of the parts is an affine geometry. 

\begin{lemma}\label{technical}
	Let $M = (E,G)$ be a full-rank, $n$-dimensional matroid with no induced $I_3$ or $F_7$-restriction. Let $w \in G-E$, 
	let $H$ be a hyperplane of $G$ that does not contain $w$, and for each $i \in \{0,1,2\}$, let $B_i = \{x \in H\colon |\{x,x+w\} \cap E| = i\}$. If $B_1$ and $B_2$ are both nonempty, then $(B_i,H) \cong \AG(n-2,2)$ for some $i \in \{1,2\}$.
\end{lemma}
\begin{proof}
	Suppose that $(B_i,H) \not\cong \AG(n-2,2)$ for $i=1$ and $i=2$. Let $T$ be a triangle of $H$. If $T$ intersects each of $B_0,B_1,B_2$, then the plane $\cl(T \cup \{w\})$ contains exactly three elements of $M$ and these elements are not a triangle, so $M$ has an induced $I_3$-restriction, contrary to hypothesis. Thus 
	\begin{claim}\label{tripletriangles}
		No triangle of $H$ intersects all three of $B_0,B_1,B_2$. 
	\end{claim}
	We now make a claim restricting the structure of certain planes.
	\begin{claim}\label{denseplane}
		If $P$ is a plane of $H$ with $P \subseteq B_1 \cup B_2$, then $P \cap B_2$ is a triangle.
	\end{claim}
	\begin{proof}[Subproof:]
		
		Suppose not. Let $x_1, \dotsc, x_7$ be an ordering of $P$ for which $\{x_1,x_2,x_3\}$ contains a basis for $P \cap B_1$, and $x_7 = x_1+x_2+x_3$. Using the fact that $P -B_1$ is not a triangle, it is routine to check that whenever $i \ge 4$ and $x_i \in B_1$, we have $x_i = x_j + x_{j'}$ for some $j < j' < i$ such that $x_j,x_{j'} \in B_1$. 
		
		For each $1 \le i \le 7$, let $Z_i = \{x_i,w+x_i\}$ and $E_i = Z_i \cap E$. By hypothesis, we have $|E_i| = 1$ if $x_i \in B_1$ and $E_i = Z_i$ otherwise. The sets $Z_i$ form a partition of $\cl(P \cup \{w\}) - \{w\}$. We will show that there is a transversal $Y$ of $E_i \colon 1 \le i \le 7$ for which $r(Y) = 3$.
		
		Let $Y_3$ be any transversal of $E_1,E_2,E_3$. Since $r(\{x_1,x_2,x_3,w\}) = 4$, it is clear that $r(Y_3) = 3$ and that $w \notin \cl(Y_3)$. Let $k \in \{3,\dotsc,7\}$ be maximal so that there is a transversal $Y_k$ of $E_1,\dotsc,E_k$ for which $\cl(Y_k) = \cl(Y_3)$. If $k = 7$, then $Y_7$ is the required $Y$; suppose that $k \le 6$.
		
		Since $w \notin \cl(Y_k)$, the plane $\cl(Y_k)$ intersects the triangle $Z_{k+1} \cup \{w\}$ in a single element $y \in Z_{k+1}$. If $y \in E_{k+1}$ then, since $Y_k \cup \{y\}$ is a transversal of $E_1, \dotsc, E_{k+1}$ that is contained in $\cl(Y_k)$, we have a contradiction to the maximality of $k$. 
		
		If $x_{k+1} \in B_2$ then $y \in Z_{k+1} = E_{k+1}$, giving this contradiction. If $x_{k+1} \in B_1$ then $x_{k+1} = x_{j} + x_{j'}$ for some $j < j' < k+1$ with $\{x_{k+1},x_j,x_{j'}\} \subseteq B_1$; note that $E_j \cup E_{j'} \subseteq Y_k$. The plane $Q = \cl(\{x_j,x_{j'},w\})$ satisfies $Q \cap E = E_j \cup E_{j'} \cup E_{k+1}$ which is a three-element set; since $(Q \cap E,Q) \not \cong I_3$ it follows that $Q \cap E$ is a triangle and so $E_{k+1} \subseteq \cl(E_j \cup E_{j'}) \subseteq \cl(Y_k)$. Thus $\varnothing \ne E_{k+1} \subseteq Z_{k+1} \cap \cl(Y_k) = \{y\}$, giving $y \in E_{k+1}$, again a contradiction.
		
		Therefore the transversal $Y$ exists, but now $Y$ is a seven-element rank-$3$ subset of $E$, contradicting the fact that $M$ has no $F_7$-restriction. Therefore $P \cap B_2$ is a triangle. 
	\end{proof}
	
	\begin{claim}\label{biclawfree} 
		Neither $(B_1,H)$ nor $(B_2,H)$ has an induced $I_3$-restriction. 
	\end{claim}
	\begin{proof}[Subproof:]
		Suppose that $(B_i \cap P,P) \cong I_3$ for some plane $P$ of $H$, where $i \in \{1,2\}$. Let $B_i \cap P = \{x_1,x_2,x_3\}$ and $z = x_1 + x_2 + x_3$, so $z \in B_j$ for some $j \ne i$. Now ~\ref{tripletriangles} gives $\{x_1+x_2,x_1+x_3,x_2+x_3\} \subseteq B_j$, as each of its elements is in a triangle with $z$ and some $x_i$. Therefore $P \subseteq B_i \cup B_j$; note that $B_i \cap P$ is linearly independent and $|B_j \cap P| = 4$, so neither $B_i \cap P$ nor $B_j \cap P$ is a triangle of $P$. If $\{i,j\} = \{1,2\}$ then we obtain a contradiction to ~\ref{denseplane}. Otherwise, $j = 0$. If $i = 1$ then $E \cap \cl(P \cup \{w\})$ is a transversal of $\{x_1,x_1+w\}, \{x_2,x_2+w\}$ and $\{x_3,x_3+w\}$ and so is a three-element independent set, giving an induced $I_3$-restriction of $M$. If $i = 2$ then $E \cap \cl(P \cup \{w\}) = \cup_{1 \le i \le 3} \{x_i,x_i+w\}$ and so $E \cap P = \{x_1,x_2,x_3\}$ gives an induced $I_3$-restriction of $M$. In either case, we have a contradiction.
	\end{proof}
	
	\begin{claim}\label{b1ag}
		$(B_1,H)$ is triangle-free. 
	\end{claim}
	\begin{proof}[Subproof:]
		If $B_1$ contains a triangle $T$ of $H$, then let $z \in B_2$ and let $P = \cl(T \cup \{z\})$. By ~\ref{tripletriangles}, for each $x \in T$, we have $x + z \in B_1 \cup B_2$; thus $P \subseteq B_1 \cup B_2$. But $P \cap B_1$ contains a triangle and so $P \cap B_2$ does not; this contradicts ~\ref{denseplane}. 
	\end{proof}
	
	It follows from ~\ref{biclawfree}, \ref{b1ag} and Corollary~\ref{recag2} that $(B_1,\cl(B_1))$ is an affine geometry. Let $F = \cl(B_1)$, so $F-B_1$ is a hyperplane of $F$. If $F = H$ then $(B_1,H) \cong \AG(n-2,2)$, giving a contradiction. Otherwise $F \ne H$, so $r(B_1) < n-1$. Since $E \subseteq \cl(B_1 \cup B_2 \cup \{w\})$ and $r(M) = n$, we also have $r(B_1 \cup B_2) \ge n-1 > r(B_1)$, so there is some $y \in B_2 - F$. 
	
	\begin{claim}\label{b2trifree}
		$(B_2,H)$ is triangle-free. 
	\end{claim}
	\begin{proof}[Subproof:]Let $T \subseteq B_2$ be a triangle. If $T \cap F = \varnothing$ then let $x \in B_1$ and let $P = \cl(T \cup \{x\})$. Since $P \cap B_1 \subseteq P \cap F = \{x\}$ and each element of $P - (T \cup \{x\})$ is on a triangle containing $x$ and an element of $T$, ~\ref{tripletriangles} gives that $P - \{x\} \subseteq B_2$. Now $P$ contradicts \ref{denseplane}. 
		
		If $|T \cap F| \ge 1$ then let $z \in T \cap F$. Since $z$ belongs to the hyperplane $F-B_1$ of $F$, there is a triangle $T'$ of $F$ containing $z$ and two elements $u_1,u_2$ of $B_1$. Let $P = \cl(\{y\} \cup T')$. Let $u_1' = y + u_1$  and define $u_2',z'$ analogously, noting that $\{u_1',u_2',z'\} \subseteq H-F \subseteq B_0 \cup B_2$. By ~\ref{tripletriangles} we have $\{u_1',u_2'\} \subseteq B_2$. If $z' \in B_0$ then the triangle $\{z',u_1',u_2\}$ contradicts \ref{tripletriangles}. If $z' \in B_2$ then the plane $P$ contains five elements of $B_2$, contradicting \ref{denseplane}. 		
	\end{proof}
	Again by Lemma~\ref{recag2}, the matroid $(B_2,\cl(B_2))$ is an affine geometry. Let $z \in B_1$ and let $z' = z+y$. Since $z \in F$ and $y \in B_2-F$ we have $z' \in G - F \subseteq B_0 \cup B_2$ so $z' \in B_2$ by \ref{tripletriangles}, so $z \in \cl(\{y,z'\}) \subseteq \cl(B_2)$. Therefore $B_1 \subseteq y + B_2 \subseteq \cl(B_2)$, giving $n-1 = r(B_2 \cup B_1) = r(B_2)$, so $(B_2,H) \cong \AG(n-2,2)$, a contradiction. 
\end{proof}

We now restate and prove Theorem~\ref{maintech}. The proof uses the description of $\cE_3$ that is given by Theorem~\ref{e3structure}, which we also include in the statement.
\begin{theorem}\label{maintechrepeat}
	If $M$ is a full-rank matroid, then $M$ contains no induced $I_3$ or $F_7$ restriction if and only if either
	\begin{enumerate}[(1)]
		\item\label{mainoutcome2} $M \in \cE_3$ (equivalently, $M$ arises from a Bose-Burton geometry of order $1$ or $2$ by a sequence of semidoublings), or 			
		\item\label{mainoutcome1} there is some $t \ge 3$ such that $M$ arises from $\AGpp(t-1,2)$ by a sequence of doublings.
	\end{enumerate} 
\end{theorem}
\begin{proof}
	Let $M = (E,G)$ and $n = r(M) = \dim(M)$. If $M$ satisfies (\ref{mainoutcome1}) or (\ref{mainoutcome2}), then it follows from Lemma~\ref{doublinggood} or the definition of $\cE_3$ that $M$ has no induced $I_3$ or $F_7$ restriction. It remains to prove the converse. The theorem is trivial for $n = 2$; suppose inductively that $n \ge 3$ and that the result holds for smaller $n$. 
	
	Let $P_5$ denote the three-dimensional matroid with five elements.  Note that if $M|P \cong P_5$ for some plane $P$, then $P$ contains a triangle $T$ with $|T \cap E| = 2$. Indeed, since $M$ is full-rank and $F_7$-free, it is not a projective geometry, so $G$ has a triangle $T$ with $|T \cap E| = 2$, even if $M$ is $P_5$-free. Let $T$ be a triangle of $G$ for which $|T \cap E| = 2$ and, if possible, so that $M|P \cong P_5$ for some plane $P$ of $G$ containing $T$. 
	Let $T = \{u,v,w\}$ where $T \cap E = \{u,v\}$. 
	
	Let $J$ be a basis of $M$ containing $\{u,v\}$ and let $H = \cl(J-\{u\})$; note that $w \notin H$ and that $r_M(E \cap H) = |J-\{u\}| = n-1$, so the matroid $M_0 = M|H$ is full-rank. Let $B_i = \{x \in H \colon |\{x,x+w\} \cap E| = i\}$ for each $i \in \{0,1,2\}$. By construction, we have $v \in B_2$. If $B_1 = \varnothing$, then $M_0 = (B_2,H)$ and $M$ is the doubling of $M_0$. Inductively, $M_0$ satisfies (\ref{mainoutcome2}) or (\ref{mainoutcome1}), and, since $\cE_3$ is closed under doublings, so does $M$. We may thus assume that $B_1 \ne \varnothing$ and so, by Lemma~\ref{technical}, that $(B_i,H) \cong \AG(n-2,2)$ for some $i \in \{1,2\}$. If $i = 2$ then $M$ contains $B_2 \cup (B_2 + w) \cong \AG(n-1,2)$ as a restriction, so Lemma~\ref{recogniseagplus} gives (\ref{mainoutcome1}). Assume therefore that $i = 1$. If $M$ has an induced $P_5$-restriction, then by choice of $T$, there is such a restriction of the form $M|P$ where $T \subseteq P$; since $w \in P-E$, the set $P \cap H$ is a triangle containing two elements of $B_2$ and one element of $B_1$, which is incompatible with $(B_1,H) \cong \AG(n-2,2)$. Therefore $M$ has no induced $P_5$-restriction. 
	
	Let $H_0$ be the $(r-2)$-dimensional flat $H-B_1$. Now $H_0 = B_0 \cup B_2$, so for every $x \in H_0$, we have $|E \cap \{x,x+w\}| \in \{0,2\}$ and for every $x \in H-H_0 = B_1$ we have $|E \cap \{x,x+w\}| = 1$. It follows that $M$ is the semidoubling of $M_0$ with respect to $H_0$.
	
	Inductively, $M_0$ satisfies (\ref{mainoutcome2}) or (\ref{mainoutcome1}). If (\ref{mainoutcome2}) holds for $M_0$ then it holds for $M$. If (\ref{mainoutcome1}) holds for $M_0$ then $M$ has an induced $\AGpp(t-1,2)$-restriction for some $t \ge 3$; any such matroid has an induced $P_5$-restriction, contradicting our argument that $M$ had no such restriction.
\end{proof}
We remark that, since a matroid $M = (E,G)$ is $(I_3,F_7)$-free if and only if $(E,\cl(E))$ is, this theorem can easily be extended to apply to matroids that are not full-rank. 
Note that outcome (\ref{mainoutcome1}) above implies that $\chi(M) = \chi(\AGpp(t-1,2)) = 2$ by Lemma~\ref{doublinggood}. Using this fact, Theorem~\ref{main2} follows from Theorems~\ref{maintechrepeat} and~\ref{hungry}. Theorem~\ref{main1} follows from Theorem~\ref{maintechrepeat} and one more lemma.

\begin{lemma}If $M \in \cE_3$ and $\chi(M) \ge 3$ then $M$ has an induced $K_5$-restriction. 
\end{lemma}	
\begin{proof}
	Let $H$ be a hyperplane of $M = (E,G)$. By Corollary~\ref{findsemidoubler1} there is some hyperplane $H'$ of $H$ and some $w \in E-H$ for which $M$ is the semidoubling of $M|H$ by $w$ with respect to $H'$. By Lemma~\ref{pushchi} we have $\chi(M|H') = \chi(M)-1 \ge 2$. If $M|H'$ is triangle-free then Corollary~\ref{recag2} implies that $(E \cap H',\cl(E \cap H'))$ is an affine geometry and so $\chi(M|H') = 1$, a contradiction; thus $M|H'$ contains a triangle $T$ of $H'$. Let $P$ be a plane of $H$ with $P \cap H' = T$. Note that $M|\cl(P \cup \{w\})$ is the semidoubling of $M|P$ by $w$ with respect to $T$. Since $|T \cap E|$ is odd and $|P \cap E|$ is even we have $|(P-T) \cap E| \in \{1,3\}$, so $M|P$ is isomorphic to either $K_4$ or $\overline{I_3}$. It is routine to check (using the fact that $K_5$ is the complement of a $5$-cycle in $\PG(3,2)$) that semidoubling either of these matroids with respect to a triangle contained in the ground set gives $K_5$; thus $M|\cl(P \cup \{w\}) \cong K_5$ as required. 
\end{proof}


\section*{Acknowledgments} 
We thank both the anonymous referees for their careful reading and helpful suggestions. 

\bibliographystyle{amsplain}


\begin{aicauthors}
\begin{authorinfo}[Bonamy]
  Marthe Bonamy\\ 
  CNRS Researcher\\
  LaBRI, University of Bordeaux\\
  Bordeaux, France\\
  marthe.bonamy\imageat{}labri\imagedot{}fr \\
  \url{http://www.labri.fr/perso/mbonamy/}
\end{authorinfo}
\begin{authorinfo}[Kardos]
  Franti\v{s}ek Kardo\v{s}\\
  Associate Professor\\
  LaBRI, University of Bordeaux\\
  Bordeaux, France\\
  fkardos\imageat{}labri\imagedot{}fr \\
  \url{http://www.labri.fr/perso/fkardos/}
\end{authorinfo}
\begin{authorinfo}[Kelly]
  Tom Kelly\\
  Ph.D. Student\\
  University of Waterloo\\
  Waterloo, Ontario, Canada\\
  t9kelly\imageat{}uwaterloo\imagedot{}ca\\
  \url{https://www.math.uwaterloo.ca/~t9kelly/}
\end{authorinfo}
\begin{authorinfo}[Nelson]
  Peter Nelson\\
  Assistant Professor\\
  University of Waterloo\\
  Waterloo, Ontario, Canada\\
  apnelson\imageat{}uwaterloo\imagedot{}ca\\
  \url{https://www.math.uwaterloo.ca/~apnelson/}
\end{authorinfo}
\begin{authorinfo}[Postle]
	Luke Postle\\
	Assistant Professor\\
	University of Waterloo\\
	Waterloo, Ontario, Canada\\
	lpostle\imageat{}uwaterloo\imagedot{}ca\\
	\url{https://www.math.uwaterloo.ca/~lpostle/}
\end{authorinfo}
\end{aicauthors}

\end{document}